\documentclass{amsart}
\newcommand{\R}{\mathbb{R}}
\renewcommand{\log}{\mathrm{ln}\,}
\newtheorem{Theo}{Theorem}
\newtheorem{Prop}{Proposition}
\begin{document}
\title{Optimal version of the Picard-Lindel\"of theorem}
\author[J.-C. Schlage-Puchta]{Jan-Christoph Schlage-Puchta}
\begin{abstract}
Consider the differential equation $y'=F(x,y)$. We determine the weakest possible upper bound on $|F(x,y)-F(x,z)|$ which guarantees that this equation has for all initial values a unique solution, which exists globally.
\end{abstract}
\maketitle

Let $F:\R^2\rightarrow\R$ be a continuous function. The well known global Picard-Lindel\"of theorem states 
that if $F$ is Lipschitz continuous with respect to the second variable, then for every real number $y_0$, 
the initial value problem $y'=F(x, y)$, $y(0)=y_0$ has a unique solution, which exists globally. On the 
other hand the initial value problem $y'=2\sqrt{|y|}$, $y(0)=0$ has infinitely many solutions, which can be parametrized 
by real numbers $-\infty\leq a\leq b\leq \infty$ as
\[
y=\begin{cases} -(x-a)^2, & x<a,\\ 0, & a\leq x\leq b,\\ (x-b)^2, & x>b.\end{cases}
\]
We conclude that uniqueness does not hold in general without the Lipschitz condition. Similarly the initial value problem 
$y'=1+y^2$, $y(0)=0$ has the solution $\tan x$, which does not exist globally. Thus, global existence also 
needs some kind of Lipschitz condition. Here we show that while some condition is necessary, being Lipschitz 
is unnecessarily strict, and determine the optimal condition. We prove the following.
\begin{Theo}
\label{thm:main}
Let $\varphi:[0, \infty]\rightarrow(0, \infty)$ be a non-decreasing function. Then the following are equivalent.
\begin{enumerate}
\item[(i)] The series $\sum_{n=1}^\infty \frac{1}{\varphi(n)}$ diverges;
\item[(ii)] For every continuous functions $F:\R^2\rightarrow\R$ for which there exists a continuous function $\psi:\R\rightarrow(0, \infty)$ such that 
\begin{equation}
\label{eq:Lipschitz}
|F(x, y)-F(x, z)|<(z-y)\psi(x)\varphi(|\log(z-y)|)
\end{equation}
holds for all real numbers $x, y, z$ such that $y<z\leq y+1$,
the initial value problem $y'=F(x, y)$, $y(0)=y_0$ has a unique local solution.
\item[(iii)] For every continuous functions $F:\R^2\rightarrow\R$ for which there exists a continuous function $\psi:\R\rightarrow(0, \infty)$ such that 
\begin{equation}
\label{eq:growth}
|F(x, y)|<|y|\psi(x)\varphi(\log(2+|y|))
\end{equation}
holds for all $x$ and $y$, every local solution of the initial value problem $y'=F(x, y)$, $y(0)=y_0$, where $y_0$ is arbitrary, can be continued to a global solution.
\end{enumerate}
\end{Theo}
In particular it is not possible to prove a general Picard-Lindel\"off type theorem with a bound that is strictly weaker than (\ref{eq:Lipschitz}) or (\ref{eq:growth}) for a function $\varphi$ satisfying (i). In this sense our theorem is indeed optimal. It might still be possible to prove existence or uniqueness under weaker conditions on the growth of $F$, if we impose other additional constrictions. However, the counterexamples we construct to prove (ii)$\Rightarrow$(i) and (iii)$\Rightarrow$(i) involve quite well behaved functions $F$, so it is not clear how such an additional assumption could look like.

Cid and Pouso \cite[Theorem~1.2]{partial} gave a quite ingenious proof for a uniqueness theorem, which is equivalent to the implication (i)$\Rightarrow$(ii) of our theorem, provided that $F(x, y_0)=0$ for all $x$ in a neighbourhood of 0. Rudin \cite{Rudin} showed that if every global solution of $y'=F(x,y)$ is unique, then there exists a function $h$ such that for all $y_0$ there exists some $x_0$, such that the solution of $y'=F(x,y)$, $y(0)=y_0$, satisfies $|y(x)|\leq h(x)$ for all $x>x_0$, whereas if solutions are not unique, then there might exist arbitrary fast growing solutions. Although this result is only loosely connected to our theorem, this work is relevant here, because the construction of the counterexamples in \cite{Rudin} is quite similar to our construction. We would like to thank the referee for making us aware of these publications.

The usual proof of the Picard-Lindel\"of theorem uses contraction on a suitably defined Banach space. For extensions of the Picard-Lindel\"of theorem using contractions we refer the reader to \cite{Feng}, \cite{Hayek} and \cite{Morci}. A different generalization was given in \cite{Sieg}. However, our proof is more elementary, once some local existence result is available. We will use Peano's theorem in the following form.
 
 \begin{Theo}[Peano]
 Let $F:\R^2\rightarrow\R$ be a continuous function, $y_0$ be some real number. Then there exists some $\epsilon>0$ and a differentiable function $y:(-\epsilon, \epsilon)\rightarrow\R$, which satisfies $y'=F(x,y), y(0)=y_0$ .
 \end{Theo}

We begin with the implication (i)$\Rightarrow$(ii) in the special case that the zero function is a solution, that is, $F(x,0)=0$ holds for all $x$. Let $y$ be a function satisfying $y'=F(x,y)$, $y(0)=1$. We claim that $y(x)>0$ for all $x>0$. If a solution tends to $+\infty$ in finite time without attaining negative values we say that this statement is also satisfied. For $n\geq 1$ define $x_n$ to be the smallest positive solution of the equation $y(x_n)=e^{-n}$. If there is some $n$ such that this equation is not solvable, then $y(x)> e^{-n}$ for this particular $n$ and all $x>0$, and our claim is trivially true, henceforth we assume that this equation is solvable for all $n$. Define $x_n^+$ to be the largest solution of the equation $y(x)=e^{-n}$ with $x\in[x_n, x_{n+1}]$. Clearly, $x_n^+$ exists. In the interval $[x_n^+, x_{n+1}]$ we have $y(x)\in[e^{-(n+1)}, e^{-n}]$, thus,
\begin{eqnarray*}
x_{n+1}-x_n^+ & \geq & \frac{e^{-n}-e^{-(n+1)}}{\max\limits_{x\in[x_n^+, x_{n+1}]} |y'(x)|}\\
& \geq & \frac{e^{-n}-e^{-(n+1)}}{\max\limits_{x\in[x_n^+, x_{n+1}]}\max\limits_{t\in[e^{-(n+1)}, e^{-n}]} |F(x, t)|} \\
& \geq & \frac{e^{-n}-e^{-(n+1)}}{e^{-n}\varphi(n+1)\max\limits_{x\in[x_n^+, x_{n+1}]}\psi(x)}\\
& \geq & \frac{1}{2\varphi(n+1)\max\limits_{x\in[0, x_{n+1}]}\psi(x)}.
\end{eqnarray*}
Assume that the sequence $(x_n)$ is bounded. Then $\max\limits_{x\in[0, x_{n+1}]}\psi(x)$ is bounded by some constant $C$. We conclude that in this case
\[
x_{n+1}-x_n\geq x_{n+1}-x_n^+\geq\frac{1}{2C\varphi(n+1)}.
\]
By assumption we have that $\sum\frac{1}{\varphi(n)}$ diverges, which contradicts the assumption that $(x_n)$ is bounded. Hence, $(x_n)$ tends to infinity. By the definition of $x_n$ we have $y(x)>0$ in $[0, x_n]$, and our claim follows.

Next suppose that $F(x,0)=0$ holds for all $x$, and $y_1$ is a solution of $y'=F(x,y)$ with $y(0)\neq 0$. Then $\tilde{y}=\frac{y_1}{y_1(0)}$ is a solution of $y'=\frac{1}{y_1(0)}F(x,y)$. As $\tilde{y}(0)=1$, we conclude that $\tilde{y}(x)>0$ holds for all $x>0$, and therefore $y_1(x)\neq 0$ for all $x>0$. 

Now suppose that $F$ satisfies the assumption of the theorem, and $y_1, y_2$ are solutions of $y'=F(x,y)$ with $y_1(0)\neq y_2(0)$. Then we consider the differential equation
\[
y' = F(x, y+y_1(x))-y_1'(x)
\]
The constant function $y=0$ is a solution. The function $y(x)=y_2(x)-y_1(x)$ is also a solution, as for this function we have
\[
F(x, y_2(x)-y_1(x)+y_1(x))-y_1'(x) = F(x, y_2(x))-y_1'(x) = y_2'(x)-y_1'(x).
\]
The function $G(x,t)=F(x, t+y_1(x))-y_1'(x)$ is continuous, and satisfies 
\begin{eqnarray*}
|G(x, t_1)-G(x, t_2)| & = & |F(x, t_1+y_1(x))-F(x, t_2+y_1(x))|\\
 & \leq & (t_1-t_2)\psi(x)\varphi(|\log(t_1-t_2)|)
\end{eqnarray*}
as well as $G(x,0)=F(x, y_1(x))-y_1'(x)=0$. In particular we know that the claimed implication holds for $G$, and we obtain that $y_1-y_2$ does not vanish. As we may revert time, it follows that solutions are unique.

Now we prove the implication (i)$\Rightarrow$(iii). 
By symmetry it suffices to consider the range $[0, \infty)$. Let $I\subseteq[0, \infty)$ be the maximal range of a solution. By Peano's theorem we know that solutions exist locally, that is, $I$ is half open. Suppose $I=[0, x_{\max})$ with $x_{\max}<\infty$. A computation similar to the one used for uniqueness shows that $y$ is bounded on $[0, x_{\max})$. As $\psi$ is continuous and $[0, x_{\max}]$ is compact, $\psi$ is also bounded. Put $Y=\sup_{x\leq x_{\max}} |y(x)|\psi(x)$. Then $F$ is bounded on $[0, x_{\max}]\times[-Y, Y]$, that is, $y$ is Lipschitz continuous on $[0, x_{\max})$, and we can extend $y$ continuously to $[0, x_{\max}]$. Moreover, as $F$ is continuous, this extension satisfies the differential equation in $x_{\max}$ if we interpret the derivative as a one-sided derivative. By Peano there exists a local solution around $x_{\max}$, which by the uniqueness we already know coincides with $y$ for $x<x_{\max}$, hence, $y$ can be extended beyond $x_{\max}$ as a solution of the differential equation. This contradicts the definition of $x_{\max}$, and we conclude that $x_{\max}=\infty$, that is, $y$ exists globally.

We now turn to the reverse implications. For a given function $\varphi$, such that $\sum_{n=1}^\infty\frac{1}{\varphi(n)}$ converges, we construct functions $F$ which satisfy the conditions (\ref{eq:Lipschitz}) resp. (\ref{eq:growth}), but for which the solutions of the corresponding differential equation are not unique resp. tend to infinity. We claim that we may assume without loss of generality that $\varphi(n)\leq n^2$. In fact, $\sum_{n=1}^\infty\frac{1}{\varphi(n)}$ converges if and only if $\sum_{n=1}^\infty\frac{1}{\min(\varphi(n), n^2)}$ converges, hence, replacing $\varphi(n)$ by $\min(\varphi(n), n^2)$ does not change condition (i), whereas conditions (ii) and (iii) become weaker.

Next we show (iii)$\Rightarrow$(i). Suppose that $\sum_{n=1}^\infty\frac{1}{\varphi(n)}$ converges. Let $F$ be a piecewise linear function satisfying $F(0)=1$, $F(e^n)=e^n\varphi(n)$. Let $y$ be a solution of the differential equation $y'=F(y)$, $y(0)=0$, and let $x_n$ be the positive solution of the equation $y(x)=e^n$. Note that $x_n$ exists and is unique, as $y'\geq 1$ for all $x\geq 0$. Then we have
\[
x_{n+1}-x_n \leq \frac{e^{n+1}-e^n}{\min\limits_{x_n\leq x\leq x_{n+1}} y'(x)} =  \frac{e^{n+1}-e^n}{\min\limits_{e^n\leq y\leq e^{n+1}} F(y)} \leq \frac{e^{n+1}-e^n}{F(e^n)} =\frac{e^{n+1}-e^n}{e^n\varphi(n)} \leq \frac{2}{\varphi(n)}.
 \]
As $\sum_{n=1}^\infty\frac{1}{\varphi(n)}$ converges, we conclude that the sequence $x_n$ converges to some finite limit $x_\infty$, that is, $y(x)$ tends to infinity as $x\rightarrow x_\infty$. We conclude that if $\sum_{n=1}^\infty\frac{1}{\varphi(n)}$ converges, then there exists a differential equation as in (iii) which does not have a global solution.

Now consider the implication (ii)$\Rightarrow$(i). Suppose that $\sum_{n=1}^\infty\frac{1}{\varphi(n)}$ converges, and let $F$ be the function satisfying $F(t)=0$ for $t\leq 0$, $F(t)=\varphi(0)$ for $t\geq \frac{1}{e}$, $F(e^{-n})=e^{-n}\varphi(n-1)$, which is continuous and linear on all intervals $(e^{-n-1}, e^{-n})$. We claim that $F$ satisfies (2). As $F$ is constant outside $[0, \frac{1}{e}]$, it suffices to check the case $0\leq y<z\leq \frac{1}{e}$. 

If $y=0$, let $n$ be the unique integer satisfying $e^{-n-1}<z\leq e^{-n}$. Then we have
\begin{eqnarray*}
|F(y)-F(z)| & = & F(z)\\
 & = & \frac{e^{-n}-z}{e^{-n}-e^{-n-1}}e^{-n-1}\varphi(n) + \frac{z-e^{-n-1}}{e^{-n}-e^{-n-1}}e^{-n}\varphi(n-1)\\
 & \leq & \frac{e^{-n}-z}{e^{-n}-e^{-n-1}}e^{-n-1}\varphi(n) + \frac{z-e^{-n-1}}{e^{-n}-e^{-n-1}}e^{-n}\varphi(n)\\
 & = & z\varphi(n)\\
 & \leq & z\varphi(|\log z|).
\end{eqnarray*}
If $y>0$, let $m\leq n$ be the unique integers satisfying $e^{-n-1}<y\leq e^{-n}$, $e^{-m-1}<z\leq e^{-m}$. If $m<n$, then 
\begin{eqnarray*}
|F(y)-F(z)| & = & F(z)- F(y)\\
 & = & \frac{e^{-m}-z}{e^{-m}-e^{-m-1}}e^{-m-1}\varphi(m) + \frac{z-e^{-m-1}}{e^{-m}-e^{-m-1}}e^{-m}\varphi(m-1)\\
&&-\frac{e^{-n}-y}{e^{-n}-e^{-n-1}}e^{-n-1}\varphi(n) - \frac{y-e^{-n-1}}{e^{-n}-e^{-n-1}}e^{-n}\varphi(n-1)\\
& \leq&z\varphi(m)-y\varphi(n-1)\\
& \leq & (z-y)\varphi(m)\\
& \leq & (z-y)\varphi(|\log z|)\\
& \leq & (z-y)\varphi(|\log(z-y)|),
\end{eqnarray*}
and our claim follows.
If $m=n$, then
\begin{eqnarray*}
|F(y)-F(z)| & = & F(z)- F(y)\\
 & = & \frac{e^{-m}-z}{e^{-m}-e^{-m-1}}e^{-m-1}\varphi(m) + \frac{z-e^{-m-1}}{e^{-m}-e^{-m-1}}e^{-m}\varphi(m-1)\\
&&-\frac{e^{-n}-y}{e^{-n}-e^{-n-1}}e^{-n-1}\varphi(n) - \frac{y-e^{-n-1}}{e^{-n}-e^{-n-1}}e^{-n}\varphi(n-1)\\
& = &(z-y)\frac{e^{-m}\varphi(m-1)-e^{-m-1}\varphi(m)}{e^{-m}-e^{-m-1}}\\
 & \leq & (z-y)\varphi(m)\\
 & \leq & (z-y)\varphi(|\log(z-y)|)
\end{eqnarray*}
We find that (2) holds in all cases.

Now consider the differential equation $y'=-F(y)$. This equation has the obvious solution $y=0$. Now consider the solution with starting value $y(0)=\frac{1}{e}$. As $y'(x)<0$ for all $x$ with $y(x)>0$, there is for every $n$ a unique $x_n$ solving the equation $y(x)=e^{-n}$. We have
\[
x_{n+1}-x_n \leq \frac{e^{-n}-e^{-n-1}}{\min\limits_{x_n\leq x\leq x_{n+1}} y'(x)} =\frac{e^{-n}-e^{-n-1}}{F(e^{-n-1})} = \frac{e-1}{\varphi(n)}.
\]
As $\sum\frac{1}{\varphi(n)}$ converges, the sequence $(x_n)$ converges to some limit $x_\infty$, and we obtain that $y(x)=0$ for $x>x_\infty$. Reversing time we find that the equation $y'=F(y)$, $y(0)=0$ does not have a unique solution. Hence, if (i) fails, so does (ii), and the proof of the theorem is complete.

We remark that the proof not only yields global existence and uniqueness of solutions, but also gives explicit bounds. Here an explicit measure for uniqueness is a bound how quickly different solutions can diverge. Equivalently we can revert time and ask how quickly solutions with different starting conditions converge. By computing the sequence $(x_n)$ occurring in the proof of the implication (i)$\Rightarrow$(ii) for specific functions $\varphi$ we obtain the following.
\begin{Prop}
\begin{enumerate}
\item[(i)] Let $F:\R^2\rightarrow\R$ be a continuous function satisfying
\[
|F(x,y)-F(x,z)|<L|y-z|
\]
and $F(x,0)=0$ for all real numbers $x, y$ and $z$. Then every solution of the equation $y'=F(x, y)$ satisfies $|y(x)|\leq e^{Lx}|y(0)|$ for all $x\geq 0$, and if $y_1, y_2$ are solutions, and $x\geq 0$, then we have
\[
|y_1(x)-y_2(x)|\geq |y_1(0)-y_2(0)| e^{-Lx}.
\]
\item[(ii)] Let $F:\R^2\rightarrow\R$ be a continuous function satisfying
\begin{equation}
\label{eq:Propsquare}
|F(x,y)-F(x,z)| < |y-z| \left(1+\sqrt{\big|\log(|y-z|)\big|}\right)
\end{equation}
for all real numbers $x, y$ and $z$ such that $y\leq z\leq y+1$.
If $y_1, y_2$ are solutions of the equation $y'=F(x,y)$ satisfying $y_1(0)-y_2(0)=1$, then we have
\[
|y_1(x)-y_2(x)|\geq e^{-x^2}
\]
for all $x>35$.
\item[(iii)] Let $F:\R^2\rightarrow\R$ be a continuous function satisfying
\begin{equation}
\label{eq:Propexpo}
|F(x,y)-F(x,z)| < |y-z| \left(1+\big|\log(|y-z|)\big|\right)
\end{equation}
for all real numbers $x,y,z$ such that $y\leq z\leq y+1$.
If $y_1, y_2$ are solutions of the equation $y'=F(x,y)$ satisfying $y_1(0)-y_2(0)=1$, then we have
\[
|y_1(x)-y_2(x)|\geq e^{-e^{2x}-4}
\]
for all $x\geq 0$.
\end{enumerate}
\end{Prop}
\begin{proof}
For the upper bound in (i) note that as $F(x,0)=0$, the Lipschitz condition with $z=0$ reads $|y'(x)|\leq L|y(x)|$, and the upper bound follows. For the lower bound note that $f(x)=y_1(x)-y_2(x)$ satisfies $|f'(x)|\leq L|f(x)|$. We may assume without loss that $y_1(0)>y_2(0)$. Then we consider $g(x)=e^{Lx}f(x)$. Put $x_0=\inf\{x:f(x)\leq 0\}$, we want to show that $x_0=\infty$. For $x\in[0, x_0]$ we have $g'(x)=e^{Lx}(Lf(x)+f'(x))\geq 0$, in particular $g(x)\geq g(0)$ and therefore $f(x)\geq e^{-Lx} f(0)$. As $f$ is continuous, we see that $x_0$ cannot be finite, and $f(x)\geq e^{-Lx}f(0)$ holds for all $x\geq 0$.

For (ii) and (iii) define $x_n$ as the least positive $x$, such that $|y_1(x)-y_2(x)| = e^{-n}$, and let $x_n^+$ be the largest real number $x$, such that $x_n\leq x\leq x_{n+1}$, and $|y_1(x)-y_2(x)| = e^{-n}$. We will give a lower bound for $x_{n+1}-x_n$, telescope these bounds to get a lower bound for $x_n$, and solve for $n$ to get a lower bound for $y_1-y_2$. 

Suppose first that $F$ satisfies (\ref{eq:Propsquare}) for all real numbers $y\leq z\leq y+1$. Then in $[x_n^+, x_{n+1}]$ we have
\begin{multline*}
|y_1'(x)-y_2'(x)|=|F(x,y_1(x))-F(x, y_2(x)| \leq
\underset{e^{-n-1}\leq |y_1-y_2|\leq e^{-n}}{\sup\limits_{x, y_1, y_2}} |F(x,y)|\\
\leq \max\limits_{e^{-n-1}\leq \delta\leq e^{-n}} \delta\left(1+\sqrt{|\log\delta|}\right) \leq e^{-n}\left(1+\sqrt{n}\right),
\end{multline*}
thus, by the mean value theorem,
\[
\frac{e^{-n}-e^{-n-1}}{x_{n+1}-x_n}\leq e^{-n}\left(1+\sqrt{n}\right),
\]
that is, $x_{n+1}-x_n\geq\frac{1-e^{-1}}{1+\sqrt{n}}$. As $x_0=0$ we obtain 
\begin{multline*}
x_n = \sum_{\nu=0}^{n-1} (x_{\nu+1}-x_\nu) \geq \sum_{\nu=0}^n\frac{1-e^{-1}}{1+\sqrt{\nu}}\\
\geq (1-e^{-1}) \int_0^n\frac{dt}{1+\sqrt{t}} = (1-e^{-1})\left(2\sqrt{n}-2\log(\sqrt{n}+1)
\right)
\end{multline*}
By the definition of $x_n$ we have $|y_1(x)-y_2(x)|>e^{-n}$ for
$0\leq x<x_n$, and we obtain $|y_1(x)-y_2(x)|>e^{-n}$ for $n\geq 3$ and
\[
x<2(1-e^{-1})(\sqrt{n}-\log n).
\]
The right hand side is larger than $1.264\sqrt{n}-1.264\log n$, and for $n>1200$ we conclude that $|y_1(x)-y_2(x)|>e^{-n}$ for $x\leq \sqrt{n+1}$. Choosing $n=\lfloor x^2\rfloor$ we obtain $|y_1(x)-y_2(x)|>e^{-n}\geq e^{-x^2}$, provided that $x>35$.

Now suppose that $F$ satisfies (\ref{eq:Propexpo}) for all real numbers 
$y\leq z\leq y+1$. Then in $[x_n^+, x_{n+1}]$ we have
\[
|y_1'(x)-y_2'(x)| = |F(x, y_1(x))-F(x, y_2(x))| \leq e^{-n}(n+1).
\]
Using the mean value theorem we obtain
\[
\frac{e^{-n}-e^{-n-1}}{x_{n+1}-x_n}\leq e^{-n}(n+1),
\]
that is, $x_{n+1}-x_n\geq \frac{1-e^{-1}}{n+1}$. As $x_0 = 0$ we obtain
\[
x_n =\sum_{\nu=0}^{n-1} (x_{\nu+1}-x_\nu) \geq \sum_{\nu=0}^{n-1} \frac{1-e^{-1}}{\nu+1} \geq (1-e^{-1})\int_1^{n+1}\frac{dt}{t} \geq (1-e^{-1})\log n
\]
If $n\geq 4$ we obtain $x_n\geq\frac{1}{2}\ln (n+1)$. Putting $n=\left\lfloor e^{2x}\right\rfloor$ we obtain $|y_1(x)-y_2(x)|\geq e^{-n}\geq e^{-e^{2x}}$ provided that $n\geq 4$, which is satisfied for $x>1$.

Since $x_4\geq (1-e^{-1})\left(1+\frac{1}{2}+\frac{1}{3}\right)\approx 1.159$, we have $|y_1(x)-y_2(x)|\geq e^{-4}$ for $x\in[0, 1]$, and therefore $|y_1(x)-y_2(x)|\geq e^{-4}\cdot e^{-e^{2x}}$ for all $x\geq 0$.
\end{proof}

The constants 35 and $e^{-4}$ have no particular meaning, we just have to capture lower order terms. We can either do so by prescribing a lower bound for $x$, as we did in (ii), or by introducing a factor as we did in (iii).

In the same way we could give upper bounds corresponding to (iii) of Theorem~\ref{thm:main}, however, it turns out that a simple ad hoc argument is much easier.

\begin{Prop}
\label{Prop:example upper}
\begin{enumerate}
\item[(i)] Let $F:\R^2\rightarrow\R$ be a continuous function satisfying $|F(x, y)|\leq |y|\sqrt{1+\log|y|}$ for all $x$ and $y$ such that $|y|\geq 1$. Then every solution of the initial value problem $y'=F(x,y)$, $y(0)=0$ satisfies $|y(x)|\leq e^{\frac{x^2}{4}+x}$ for all $x\geq 0$.
\item[(ii)] Let $F:\R^2\rightarrow\R$ be a continuous function satisfying $|F(x, y)|\leq |y|\log|y|$ for all $x$ and $y$ such that $|y|\geq e$. Then every solution of the initial value problem $y'=F(x,y)$, $y(0)=0$ satisfies $|y(x)|\leq e^{e^{x}}$ for all $x\geq 0$.
\end{enumerate}
\end{Prop}
\begin{proof}
Let $F$ and $y$ be as in (i). The function $\tilde{y}(x)=e^{\frac{x^2}{4}+x}$ satisfies the equation $y'=y\sqrt{1+\log y}$, $y(0)=1$. We claim that for all $x\geq 0$ we have $|y(x)|< \tilde{y}(x)$. Define $x_0=\sup\{x>0:|y(x)|< \tilde{y}(x)\}$. Clearly $x_0>0$. For $x\in[0, x_0)$ we have $|y(x)|\leq 1$ or 
\begin{multline*}
\tilde{y}'(x) - y'(x) = \tilde{y}(x)\sqrt{1+\log \tilde{y}(x)} - F(x, y(x)) \geq\\
 \tilde{y}(x)\sqrt{1+\log \tilde{y}(x)} - |y(x)|\sqrt{1+\log |y(x)|}>0.
\end{multline*}
If $x_0\neq\infty$, it follows that $\tilde{y}(x_0) > y(x_0)$. In the same way we obtain $\tilde{y}(x_0) > -y(x_0)$, and conclude that $x_0=\infty$.

Now let $F$ and $y$ be as in (ii). The function $\tilde{y}(x)=e^{e^x}$ satisfies the equation $y'=y\log y$, $y(0)=e$, and we obtain
\[
\tilde{y}'(x)-y'(x) = \tilde{y}(x)\log \tilde{y}(x) - F(x, y(x)) \geq \tilde{y}(x)\log \tilde{y}(x) - y(x)\log y(x)
\]
for all $x$ such that $e\leq y(x)<\tilde{y}(x)$, and our claim follows as in the first case.
\end{proof}
In general whenever one can give a lower bound for the growth of the partial sums $\sum_{n\leq N}\frac{1}{\varphi(n)}$, one obtains upper bounds for the growth of solutions and for the convergence of different solutions with different starting values.

\end{document}